\documentclass[11pt]{article}

\usepackage{amsmath,amssymb,amsthm,amscd,epsfig}
\newcommand{\vp}{\varepsilon}

\newcommand{\n}{\noindent}

\theoremstyle{plain}

\newtheorem{thm}{Theorem}

\newtheorem*{unthm}{Theorem}

\newtheorem{cor}{Corollary}

\newtheorem*{unpro}{Proposition}

\theoremstyle{definition}

\newtheorem{defn}{Definition}

\theoremstyle{remark}

\begin{document}

\title{Inverse limits and statistical properties for  chaotic implicitly defined economic models}

\author{Eugen Mihailescu}
\date{}
\maketitle

\begin{abstract}

In this paper we study the dynamics and ergodic theory of certain
economic models which are implicitly defined. We consider
1-dimensional and 2-dimensional overlapping generations models, a
cash-in-advance model, heterogeneous markets and a cobweb model
with adaptive adjustment. We consider the inverse limit spaces of
certain chaotic invariant fractal sets and their
metric, ergodic and stability properties. The inverse limits give
the set of intertemporal perfect foresight equilibria for the
economic problem considered. First we show that the inverse limits
of these models are stable under perturbations. We prove that the
inverse limits are expansive and have specification property. We
then employ utility functions on inverse limits in our case. We
give two ways to rank such utility functions. First, when
perturbing certain dynamical systems, we rank utility functions in
terms of their \textit{average values} with respect to invariant
probability measures on inverse limits, especially with respect to
measures of maximal entropy. For families of certain unimodal maps
we can adjust both the discount factor and the system parameters
in order to obtain maximal average value of the utility. The
second way to rank utility functions (for more general maps on
hyperbolic sets) will be to use equilibrium measures of these
utility functions on inverse limits; they optimize
 average values of utility functions while \textit{at the same time} keeping the
disorder in the system as low as possible in the long run.
\end{abstract}



\section{Non-invertible economic models. Outline of main results}

Non-invertible dynamical systems have found many applications in various economic models, in which the equilibrium at
 time $t+1$ is not uniquely defined by the one at time $t$; instead there may exist several such optimal states at time $t+1$. We refer to these systems as \textit{implicitly defined economic systems}.

In this paper we study the dynamical and ergodic properties of
such systems which present chaotic behavior on certain invariant
sets. Among the economic systems with non-invertible (or
\textit{backward}) dynamics there are the 1-dimensional and the
2-dimensional overlapping generations models, the cash-in-advance
model, the cobweb model with adaptive adjustment and a class of
models representing heterogeneous market agents with adaptively
rational rules. The common feature of all these models is that
they are given by non-invertible dynamical systems and present
chaotic behavior. \ In some of these models, we have
\textit{hyperbolic horseshoes} (as in the cobweb model, see
\cite{O}, \cite{Z}), in others \textit{transversal
homoclinic/heteroclinic orbits from saddle points} (see the
heterogeneous market model, \cite{FG}), or yet in others there exist
\textit{snap-back repellers}, as in the 1-dimensional and
2-dimensional overlapping generations models for certain offer
curves (see \cite{GHT}). Also in the case of unimodal maps
modelling some overlapping generations scenarios,
 we have chaotic behavior on \textit{repelling invariant Cantor sets} (as for the logistic map
 $F_\nu$ with $\nu>4$, see \cite{Med}, \cite{R}).

For such noninvertible dynamical systems, the inverse limits are
very important since they provide a natural framework in which the
system "unfolds" and they give sequences of intertemporal
equilibria. Also as we will see they are important since many
results from the theory of expansive homeomorphisms can be applied
on inverse limits, in particular those about lifts of invariant
measures. \textit{Equilibrium measures} of Holder potentials are significant examples of invariant measures and they are very important for the evolution of the system.
For instance, the measure of maximal entropy gives the
distribution on the phase space associated to "maximal chaos". The
Sinai-Ruelle-Bowen measure (see \cite{ER}, \cite{Y}) on a
hyperbolic attractor or of an Anosov diffeomorphism is again an
equilibrium measure (for the unstable potential), and gives the
limiting distribution of the forward iterates of Lebesgue-almost
all points in a neighbourhood of the attractor. Thus it is a
\textit{natural measure} or \textit{physical measure} of the
system since it can be actually observed in experiments/computer
simulations.

Another important feature for economic dynamical systems is that
of \textit{stability}. We are interested if a certain model is
\textit{stable} on invariant sets at small fluctuations. In our
case, since we work with infinite sequences of intertemporal
equilibria, one would like to have stability of the shifts on the
inverse limit spaces.

The standard method of studying evolution of a system in economics
is to use random (stochastic) dynamical systems which transfers
exogeneous random "shocks" to the system. However a system which
presents chaotic behavior, has also complicated
\textit{endogeneous} fluctuations.

Also given an implicitly defined economic system with its inverse
limit of intertemporal equilibria and an utility function on these
equilibria, a central goverment/central bank may want to find a
\textit{distribution on the set of intertemporal equilibria} which
maximizes the average value of the utility, but at the same
time keeps the disorder in the system as little as possible in the
long run. If $W(\cdot)$ is a utility function on $\hat \Lambda$
and $\hat \mu$ is a $\hat f$-invariant measure on $\hat \Lambda$
with measure-theoretic entropy $h_{\hat \mu}$, then the maximum in $\hat \mu$ of
the expression $$ \int_{\hat \Lambda} W(\hat x) d\hat \mu(\hat x)
+ h_{\hat \mu} $$ is attained for the \textit{equilibrium measure}
$\hat \mu_W$ of $W$ (see for instance \cite{KH} for the
Variational Principle for Topological Pressure). So the
equilibrium measures may provide a good way to do that, and we
will be able to give geometric and statistical properties of these
measures. One of the defining characteristics of chaos is
sensitive dependence on initial conditions, that is, even if we
start with two initial states that are quite close to each other,
still over time, they may become very far from each other. The
equilibrium measures will permit us to estimate the
\textit{measures of sets of points which stay close} up to $n$
iterations.

We will use the notion of \textbf{chaotic map} several times. We
say that $f$ is \textbf{chaotic} on an invariant set $X$ if $f$ is
topologically transitive on $X$ and $f$ has sensitive dependence
on initial conditions (see for eg. \cite{R}).

\

The \textbf{main sections and results} of the paper are the following:

First we review some important economic models with non-invertible dynamics, like the overlapping generations model,
the cash-in-advance model, the cobweb model with adaptive adjustments and the heterogeneous market model.
A common feature of all these models is the backward dynamics born out of  implicitly defined difference equations.
Also in many instances we have chaotic invariant sets for these models, given by horseshoes, or by snap-back repellers,
or by transverse homoclinic orbits. Therefore we have hyperbolicity on certain invariant sets or conjugation of an iterate
 with the shift on some 1-sided symbol space $\Sigma_m^+$.

In Theorem \ref{perturbation} we will prove that by slightly
\textbf{perturbing} the parameters of these difference equation, we 
obtain again the same dynamical properties, for instance density
of periodic points, topological transitivity, etc.

We study then \textbf{utility functions on inverse limits} for
noninvertible economic systems. Invariant measures for
 a dynamical system are very important since they preserve the ergodic and dynamical properties of the system in time;
 in fact from any measure one can form canonically an invariant measure by a well-known procedure (see for eg. \cite{KH}).
We will give \textbf{two options to rank utility functions}: \ one
using average values with respect to invariant probability
borelian measures, especially measures of maximal entropy (which
best describe the chaotic distribution of the system over time),
and another by using equilibrium measures of the utility
functions, which give the best average value while keeping the
system as under control as possible.

\textbf{The first option} is given in Theorem \ref{inv} where we rank utility functions of systems given by certain
 unimodal maps according to their average values with respect to invariant borelian measures $\hat \mu$ on the inverse limits,
 especially with respect to measures of maximal entropy.
For certain expanding systems, namely for logistic maps $F_\nu,
\nu>4$ we are able to compare in Corollary \ref{log} the
\textbf{average utility values} with respect to the corresponding
measures of maximal entropy when perturbing both the discount
factor $\beta$ of the utility $W$, as well as the system parameter $\nu >4$.

Then in Theorem \ref{exp} we will prove that the inverse limits of
certain invariant sets for these models are \textbf{expansive},
and have also the \textbf{specification property}. This will allow
us in Theorem \ref{eq} to show that given a Holder continuous
potential, we can associate to it a special probability measure
called an \textbf{equilibrium measure} (see \cite{KH}, \cite{Bo}
for definitions). This equilibrium measure can be estimated
precisely, on sets of points remaining close to each other up to a
certain positive iterate (i.e on Bowen balls). We can apply these
results to utility functions from economics, which are shown to be
 Holder potentials.

\textbf{The second option to rank utility functions} we consider,
 is to maximize the \textbf{ratio} between the exponential of the
average value with respect to $\hat \mu$ and the measure $\hat \mu$ of the set of
points from the inverse limit that remain close up to a certain number of iterates. In
this way we find the distribution $\hat \mu$ which maximizes the average
utility value but \textbf{at the same time} keeps the "disorder"
of the system (i.e the entropy of $\hat \mu$) as small as possible
(equivalently the measure of the set of points which shadow $x$ up
to order $n$, is as large as possible).   Equilibrium
measures of Holder potentials on the inverse limit have also other
various statistical properties, like \textit{Exponential Decay of
Correlations} on Holder observables (see \cite{Bo}). Then in
Theorem \ref{approx} we approximate the average value of the
utility on inverse limits with those of simpler potentials.

\

Let us remind now several examples of economic dynamical systems,
which are non-invertible:

\textbf{1. The 1-dimensional overlapping generations model.}

This model was proposed initially by Grandmont (\cite{G}) and
studied by various authors (\cite{GHT}, \cite{KS}, \cite{Med}, \cite{Med2}). In
this model we have an economy with constant population divided
into young and old agents, and with a household sector and a
production sector. A typical agent lives for the 2 periods, works
when young and consumes when old and he receives a salary for his
work in the first period. There is a perishable consumption good
and one unit of it is produced with one unit of labour. If money
is supplied in a fixed amount, say $M$, then we have at time $t$,
that $w_t \ell_t = M$, where $w_t$ is the wage rate and $\ell_t$
is the labour. At the same time, $M = p_{t+1} c_{t+1}$ where
$p_{t+1}$ is the expected price of the consumption good at time
$t+1$ and $c_{t+1}$ is the amount of future consumption.  Now
agents have an utility function of type $U = V_1(\ell_* - \ell_t)
+ V_2(c_{t+1})$ where $\ell_*$ is the fixed labour endowment of
the young and $\ell_* - \ell_t$ is the leisure at time $t$.
Agents would like to have both as much leisure currently as well
as consumption when old. Thus under the budget constraint from
above $M = w_y \ell_t = p_{t+1} c_{t+1}$ the optimization problem
above gives, by the method of Lagrange multipliers, an implicit
difference equation: \ $\ell_t = \chi(c_{t+1})$, where
$\chi(\cdot)$ is the \textit{offer curve}. Since by assumption one
unit of labour produces one unit of consumption good, we have
$\ell_t = c_t$, hence by denoting $\ell_t$ by $x_t$, we obtain
\begin{equation}\label{1OLG}
y_t = \chi(y_{t+1})
\end{equation}
As Grandmont showed in \cite{G}, in many cases the offer curve is
not given by a monotonic/injective function, making (\ref{1OLG}) a
non-invertible difference equation. Thus for a level of
consumption at time $t$ there may be several levels of optimal
consumption at time $t+1$. In this case we study the
 \textit{backward dynamics} of the system, i.e the sequences of future consumption levels allowed by
 (\ref{1OLG}).\
The backward dynamics given by relation (\ref{1OLG}) is chaotic in
certain cases. For instance a condition was given by Mitra and
extended in \cite{GHT} in order to guarantee the existence of a
\textit{snap-back repeller}. Let us first recall the definition of a
snap-back repeller (see \cite{Mar}, \cite{Mar2}), and that of the
one-sided shift:

\begin{defn}\label{snap}
Let a smooth function $f:U \to U$, where $U$ is an open set in $\mathbb R^n, n \ge 1$.
 Suppose that $p$ is a fixed repelling point of $f$, i.e all the eigenvalues of $Df(p)$ are larger than 1 in absolute
  value, and assume that there exists another point $x_0 \ne p$ in a repelling neighbourhood of $p$, so that
   $f^m(x_0)= p$ and $\text{det}Df(f^i(x_0)) \ne 0, 1 \le i \le m$. Then $p$ is called a \textit{snap-back
   repeller} of $f$.
\end{defn}

\begin{defn}\label{shift}
We will denote by $\Sigma_m^+$ (where $m \ge 2$) the space of
1-sided infinite sequences formed with $m$ symbols, i.e
$\Sigma_m^+ = \{(i_0, i_1, i_2, \ldots), i_j \in \{1, \ldots, m\},
j \ge 0\}$.  We have the \textit{shift map} on $\Sigma_m^+$,
namely $\sigma_m: \Sigma_m^+ \to \Sigma_m^+, \ \sigma_m(i_0, i_1,
\ldots) = (i_1, i_2, \ldots)$. The space $\Sigma_m^+$ is  compact
with the product topology.
\end{defn}

 Snap-back
repellers appear only for non-invertible maps, and are
important since they are similar to transverse homoclinic orbits (see \cite{KH} for eg.)
Marotto proved the following:

\begin{unthm}[Marotto]\label{marotto}
Let $p$ a snap-back repeller for a smooth non-invertible
map $f$ and $\mathcal{O}(x_0)$ a homoclinic orbit of
$x_0$ towards the repelling fixed point $p$, i.e $\mathcal{O}(x_0)
= \{\ldots, x_{-i}, \ldots, x_0, f(x_0), \ldots, p\}$, with
$f(x_{-i}) = x_{-i+1}, i \ge 1$. Then in any neighbourhood of the
orbit $\mathcal{O}(x_0)$ there exists a Cantor set $\Lambda$ on
which some iterate of $f$ is topologically conjugated to the shift
on the space $\Sigma_2^+$ of one-sided infinite sequences on 2
symbols. Hence $f$ itself is chaotic on $\Lambda$.
\end{unthm}

For many economic models, the offer curve $\chi(\cdot)$ is given by a smooth (or piecewise smooth) \textit{unimodal map}
 (see \cite{G}, \cite{GHT}, \cite{Med}, \cite{Med2}). We shall recall some of their properties; for more information, see \cite{KS}, \cite{Med}, \cite{Med2},  etc.

A continuous map $f:[a, b] \to [a, b]$ is called \textit{unimodal}
if $f$ is not monotone and there exists a point $c \in (a, b)$ so
that $f(c) \in [a, b]$ and $f$ is increasing on $[a, c)$ and
decreasing on $(c, b]$. Type A unimodal maps are unimodal maps
satisfying $f(a) = a$ and $f(c) < b$. Type B unimodal maps are
those satisfying $f(a) > a$ and $f(b) = a$. Type C maps are of the
form $f:[a, b] \to \mathbb R$ s. t $f$ is not monotone, $f(a)=f(b)
= a$ and $f(c) > b$. Type C maps are not strictly speaking
unimodal as the map $f$ does not take necessarily values inside
the same interval $[a, b]$, but in general they are considered
"unimodal" too. \ In certain cases when the offer curve $\chi$ is
unimodal, one can find snap-back repellers (see \cite{GHT}):

\begin{unpro}
Let $\chi: I \to I$ be a unimodal smooth function on the unit interval, with a maximum point at $x_m$ and a fixed point at $x^*$. If $\chi^3(x_m) < x^*$, then $x^*$ is a snap-back repeller and thus there exists an invariant Cantor set $\Lambda \subset I$ on which an iterate of $\chi$ is topologically conjugate to the shift; so $\chi$ is chaotic and has positive topological entropy.
\end{unpro}

We will need in conjunction with unimodal maps and their inverse limits, the notions of \textit{topological attractor}
 and \textit{asymptotically stable attractor}. First given a continuous map $f:X \to X$ on a metric space and a closed
  forward invariant set $K \subset X$, we call the \textit{basin of attraction} of $K$ the set
  $B(K):= \{y \in X, \omega(y) \subset K\}$, i.e the set of points having all the accumulation points
  of their iterates, contained in $K$. Then we say that $K$ is a \textit{topological attractor}, if
  $B(K)$ contains a residual set in an open neighbourhood $U$ of $K$ (i.e the complement of $B(K)$ in $U$
  is contained in a countable union of nowhere dense subsets) and if there is no closed forward invariant subset
  $K' \subset K$ s.t $B(K)$ and $B(K')$ coincide up to a countable union of nowhere dense sets.  \
   If $K$ is $f$-invariant (i.e $f(K) = K$), it has arbitrarily close neighbourhood $V$ s.t
   $f(V) \subset V$ and the basin $B(K)$ is open, then we say that $K$ is an \textit{asymptotically stable attractor}.

\begin{defn}\label{inverse-limit}
Given a continuous map $f: X \to X$ on a metric space $(X, d)$, we form
the \textbf{inverse limit} $(\hat X, \hat f)$, where $\hat X :=
\{\hat x = (x, x_{-1}, x_{-2}, \ldots), f(x_{-i}) = x_{-i+1}, i
\ge 1\}$ and $\hat f:\hat X \to \hat X, f(x, x_{-1}, \ldots) =
(f(x), x, x_{-1}, \ldots), \hat x \in \hat X$. We consider the
topology induced on $\hat X$ from the infinite product of $X$ with
itself. In fact $\hat X$ is a metric space with the metric
$$d(\hat x, \hat y) = \mathop{\sum}\limits_{i \ge 0}
\frac{d(x_{-i}, y_{-i})}{2^i}, \hat x, \hat y \in \hat X$$
\end{defn}

For a $\mathcal{C}^3$ smooth map $f$ on the interval $[a,
b]$, the \textit{Schwarzian derivative} is $Sf(x):=
\frac{f'''(x)}{f'(x)} - \frac 32 (\frac{f''(x)}{f'(x)})^2, x \in
[a, b]$.
We have then, by collecting several results (see \cite{Med}, \cite{R} and references therein) the following:

\begin{unthm}[Attractors in inverse limit spaces of unimodal maps]

 \ a) \ Let $f$ be a type A unimodal map on the interval $[0, 1]$, with $Sf <0$ on $[0, 1]$. If $f^2(c) = f(1) >0$ and
 $f'(0) >1$, then $\hat 0 = (0, 0, \ldots)$ is an asymptotically stable attractor and a topological attractor for
 $\hat f$ and it is the only topological attractor for $\hat f$.

 \ b) \ Let $f:[0, 1] \to [0, 1]$ be a unimodal map of type B with $Sf <0$ and assume that
 $f$ has a unique fixed point $p \in (c, 1]$ that is repelling for $f$ s.t $f(0) > p$.
  Then the point $\hat p = (p, p, \ldots) \in \widehat{[0, 1]}$ is an asymptotically stable attractor and a topological
   attractor for $\hat f$ and it is the only topological attractor of $\hat f$ in $\widehat{[0, 1]}$.

\ c) \ Let $f:[0, 1] \to [0, 1]$ be a unimodal map of type B with
$Sf <0$ and with $f(0) < p$, where $p$ is the unique fixed point
in $(c, 1]$. Assume that $f$ has topological attractor $P$ which
is either chaotic or periodic. Then the basin of attraction of $P$
contains a union of $n$ intervals $A_0, \ldots, A_{n-1}$ with
$f^i(A_0) \subset A_i, 1 \le i \le n-1$. Let $\Lambda$ be the set
of points in $[0, 1]$ that are never attracted to $P$. Then
$\Lambda$ is partitioned as $\Lambda_1 \cup \ldots \cup \Lambda_m$
where $\Lambda_j$ is an $f$-invariant transitive Cantor set and
$f|_{\Lambda_j}$ is conjugate to a subshift of finite type. Then
the shift map $\hat f$ has a unique topological attractor namely
$\hat \Lambda_0$.

\ d) \ Consider the Type C logistic map $F_\nu(x) = \nu x (1-x), x
\in [0, 1]$ for $\nu >4$, and let $\Lambda_\nu :=
\mathop{\cap}\limits_{n \ge 0} F_\nu^{-n}([0, 1])$. Then
$\Lambda_\nu$ is $F_\nu$-invariant and $F_\nu$ is topologically
conjugate to the shift on $\Sigma_2^+$. Also $\hat \Lambda_\nu$ is
an asymptotically stable attractor for $\hat F_\nu$.
\end{unthm}

\textbf{2. The 2-dimensional overlapping generations model.}

As in the 1-dimensional model before, we have an economy with two
sectors, a household and a production sector (see \cite{GHT}). The
household sector is the same as before, hence with perfect
foresight, we have for the offer curve $\chi(\cdot)$: \ $\ell_t =
\chi(c_{t+1})$. By comparison with the previous case, output is
now produced both from labour $\ell_t$ supplied at time $t$ by the
household sector, and by \textit{capital stock} $k_{t-1}$ from the
previous period $t-1$, supplied by non-consuming companies which
tend to maximize their profits. The output $y_t$ is the minimum
between $\ell_t$ and $k_{t-1}/a$, where $1/a$ is the productivityy
of the capital. We assume that the capital stock available at the
begining of period $t+1$ is $ k_t = (1-\delta)k_{t-1} + i_t, $
where $0<\delta<1$ is the depreciation rate of the capital and
$i_t$ is the investment, i.e the portion of the output at time $t$
which is invested in the next period. Thus the consumption at time
$t$ is $c_t = y_t - i_t$, and at equilibrium we have  $y_t = \ell_t =
\frac{k_{t-1}}{a}$. One obtains then the second order difference
equation: $$ y_t = \chi[a(1-\delta+\frac 1a)y_{t+1}-ay_{t+2}] $$
Hence by substituting $z_t = y_t$ and $w_t = y_{t+1}$ we obtain
the implicitly defined system of equations:
\begin{equation}\label{2OLG}
\left\{ \begin{array}{ll}
 z_t = \chi[a(1-\delta+ \frac {1}{a})z_{t+1} -
a w_{t+1}]  \\

w_t = z_{t+1}
 \end{array}
  \right.
\end{equation}

In this model for certain parameter values (see \cite{GHT}), the fixed point $x^*$ is a snap-back repeller,
thus by the results of Marotto (see \cite{Mar}, \cite{Mar2}) in any neighbourhood of the orbit of the snap-back repeller there is an invariant set on which $f$ is chaotic and conjugate to a 1-sided shift.

\textbf{3. Cash-in-advance model.}

The following model can be found in \cite{MR} or \cite{KSY}. In this economy there exists a central government and a representative agent, where the government consumes nothing and sets monetary policy. There exists also a cash good and a credit good, and the agent has a utility function of type
\begin{equation}\label{util}
\mathop{\sum}\limits_{t=0}^\infty \beta^t U(c_{1t}, c_{2t}),
\end{equation}
where $\beta \in (0, 1)$ is the discount factor. The function $U$
takes the form \ $ U(x, y) = \frac{x^{1-\sigma}}{1-\sigma} +
\frac{y^{1-\gamma}}{1-\gamma}, $ with $\sigma>0, \gamma >0$. The
cash good $c_{1t}$ can be bought with money $m_t$, which is
carried over from period $t-1$. The credit good $c_{2t}$ does not
require cash and can be bought on credit. Each period the agent
has an endowment $y$ and $c_{1t} + c_{2t} = y$. We assume also
that the cash good costs the same price $p_t$ as the credit good.
The agent wants to maximize his utility function by a choice of
$\{c_{1t}, c_{2t}, m_{t+1}\}_{t \ge 0}$ subject to constraints: \ $
p_t c_{1t} \le m_t$, and \ $m_{t+1} \le p_t y + (m_t -
p_tc_{1t}) + \theta M_t - p_tc_{2t},$ where $M_t$ is the money
supply controlled by the government for a constant growth, $M_{t+1}
= (1 + \theta) M_t$. Denote by $x_t = m_t/ p_t$ the level
of real money balance. We obtain then an implicitly defined difference equation giving
$x_t$ in terms of $x_{t+1}$ with the help of a non-invertible map $f$, i.e
\begin{equation}\label{CIA}
x_t = f(x_{t+1})
\end{equation}

For certain parameters, it can be shown (see \cite{MR}) that there exists an invariant interval
$[x_l, x_r]$ such that the map $f$ has a periodic cycle of period 3. Hence according to Li-Yorke classic result (see \cite{LY}),
 the map $f$ is chaotic on that interval. In fact it can be shown
 that there exists an invariant subset of $[x_l, x_r]$ on which
 the map is conjugate to a subshift of finite type.

\textbf{4. Cobweb model with adaptive adjustment-hedging.}

In this model (see \cite{O}) the supplier adjusts his production $x_t$ according to the realities of the market while keeping the intention to reach a profit maximum $\tilde x_{t+1}$. It is met for instance in agricultural markets where farmers who plant for example wheat cannot change their crop during the same year/period. This is a hedging rule
$$
x_{t+1} = x_t + \alpha (\tilde x_{t+1} - x_t),
$$
with $\alpha \in (0, 1)$ the speed of adjustment. The aggregate supply from $n$ identical producers is $X_t = nx_t$, and the price is given by $p_t = \frac{c}{Y_t^\beta}$, where $Y_t$ is the demand at period $t$ and $c$ is a fixed parameter. We assume the market clears at each period, i. e $X_t = Y_t$.
Then after a change of variable we obtain the equation
\begin{equation}\label{cobweb}
z_{t+1} = f_{\alpha, \beta}(z_t) = (1-\alpha) z_t + \frac{\alpha}{z_t^\beta}, \ (\alpha, \beta) \in (0, 1) \times (0, \infty)
\end{equation}

This function has a unique fixed point $z = 1$, which is a repeller if $|f_{\alpha, \beta}'(1)| >1$, i.e if $\beta > \frac{2-\alpha}{\alpha}$.

Then Onozaki et. al. (\cite{O}) showed that there exists a number $\bar \beta > \frac{2-\alpha}{\alpha}$ s.t for each $\beta > \bar \beta$, $f_{\cdot, \beta}(\cdot)$ has a hyperbolic horseshoe in the plane.

\textbf{5. A heterogeneous market model.}

We will give only the final formula for this 2-dimensional non-invertible case; more information can be found in \cite{FG}.
One has to study the dynamics of the non-invertible map:

\begin{equation}\label{HM}
\left\{ \begin{array}{ll}
 z_{t+1} = z_t[(1-\alpha) - \alpha\frac{b(1-m_t)}{2B}] \\

m_{t+1} = \text{tanh} [\frac{\beta b}{4} \cdot z_t^2 \cdot
(\frac{b(1-m_t)}{B}+1) + \frac {\beta}{2}(C_2 - C_1)]

 \end{array}
  \right.
\end{equation}

For this model, Foroni and Gardini proved in \cite{FG} that there
are \textit{saddle cycles} with homoclinic or heteroclinic transverse
intersections for certain parameters, which give rise to chaotic
sets (horseshoes) by Smale's Theorem or its variants (see
\cite{R}, \cite{HL}, etc.).

\textbf{Conclusions:}

In the examples above there exist parametrizations in which the
system given implicitly $z_t = f(z_{t+1})$, has some hyperbolic
set $\Lambda$ (in general without critical points) or a set where an iterate is conjugate to a 1-sided shift. The dynamics/ergodic theory in these two cases are very similar. The hyperbolic
case includes also the case with no contracting directions, i.e
the expanding case. The implicit difference equation gives the
\textit{backward dynamics} of the model. We notice that a point
from the inverse limit $\hat \Lambda$ given by $\hat x = (x,
x_{-1}, \ldots)$ represents in fact a sequence of \textit{future
equilibria} which are \textit{allowed by the backward dynamics};
so in the notation $\hat x = (x, x_{-1}, x_{-2}, \ldots)$, we
start from a level of consumption of $x$, then at time 1 we have a
level of consumption $x_{-1}$, then $x_{-2}$ at time 2, and so on.

\section{Metric and ergodic properties on inverse limits of chaotic economic models.}

For the implicitly defined economic models given before, we have seen that there exist invariant
sets on which the function (or one of its iterates) is conjugated to a shift on a symbol space;
this invariant limit set $\Lambda$ is usually obtained from homoclinic/heteroclinic orbits or snap-back repellers and thus we have
 a hyperbolic structure on $\Lambda$ (see \cite{R}, \cite{KH}, \cite{Mar2}, etc.)

  Hyperbolicity is understood here in the \textit{endomorphism sense},
   in which the unstable directions and unstable manifolds depend on whole sequences of consecutive preimages (i.e elements of $\hat \Lambda$), not only on base
   points (see \cite{Ru-carte89}, \cite{M-DCDS06} for definitions). We
   include in the hyperbolic case also the case of no contracting
   directions, i.e the expanding case. For a hyperbolic map $f$ on a compact invariant set $\Lambda$ and a small enough $\delta>0$, we denote by $W^s_\delta(x)$ the local stable manifold at the point $x \in \Lambda$, and by $W^u_\delta(\hat x)$ the local unstable manifold corresponding to the history $\hat x \in \hat \Lambda$. \
Let us prove that in this non-invertible hyperbolic case we
have stability of the  inverse limits:

\begin{thm}\label{perturbation}
Let us consider one of the economic models from Section 1, given by a dynamical system $f$ having a hyperbolic invariant set $\Lambda$. Then given any dynamical system $g$ obtained by a small $\mathcal{C}^2$ perturbation of the parameters of $f$, there exists a $g$-invariant set $\Lambda_g$ and a homeomorphism $H: \hat \Lambda \to \hat \Lambda_g$ such that $\hat g \circ H = H \circ \hat f$. Thus the dynamics of $\hat g$ on $\hat \Lambda_g$ is the same as the dynamics of $\hat f$ on $\hat \Lambda$.
\end{thm}

\begin{proof}
From the discussion and references given in Section 1 we see that each model has, for certain parameter choices,
 invariant sets obtained from homoclinic or heteroclinic orbits, snap-back repellers or horseshoes (like the cobweb model).
The hyperbolicity is obtained from Smale's Theorem on transverse
homoclinic or heteroclinic intersections (see \cite{R}) or its
non-invertible variant given by Hale and Lin (\cite{HL}).  Now let
$U$ be a neighbourhood of $\Lambda$ s.t $\Lambda =
\mathop{\cap}\limits_{n \in \mathbb Z} f^{-n}(U)$.
 Then if $g$ is obtained from $f$ by a small $\mathcal{C}^2$ perturbation, we can form the basic set $\Lambda_g = \mathop{\cap}\limits_{n \in \mathbb Z} g^{-n}(U)$. If $f$ is hyperbolic on $\Lambda$, then also $g$ will be hyperbolic on $\Lambda_g$.
The hyperbolicity is understood as for endomorphisms, since $f$ is not necessarily invertible on $\Lambda$ (for instance for $\Lambda$ obtained from a snap back repeller, there are at least two points in $\Lambda$ with $f$-image equal to the fixed repelling point).

Hence from \cite{M-DCDS06} we infer the existence of a
conjugating homeomorphism $H : \hat \Lambda \to \hat \Lambda_g$ between the
inverse limit of $(\Lambda, f)$ and that of $(\Lambda_g, g)$,
which commutes with the lifts $\hat f$ and $\hat g$.

\end{proof}

Notice also that by perturbations and by lifting to the inverse limit, the topological entropy 
is not changed, i.e $h_{top}(g|{\Lambda_g}) = h_{top}(\hat g|_{\hat \Lambda_g}) = 
h_{top}(f|_\Lambda) = h_{top}(\hat f|_{\hat \Lambda})$. \
We discuss now the notion of \textit{utility function} on the set of intertemporal equilibria (see for eg. \cite{KS}, \cite{S}).

\begin{defn}\label{ut}
Consider a continuous function $f: X \to X$ which is non-invertible on the compact set $X$ contained in $\mathbb R$ or $\mathbb R^2$, and let $\hat X$ be the inverse limit.
A \textbf{utility function} on $\hat X$ is a function $W: \hat X \to \mathbb R$ given
by $$ W(\hat x) = \mathop{\sum}\limits_{i \ge 0} \beta^i
U(x_{-i}), $$ where $\beta \in (0, 1)$ is called the \textit{discount
factor} and,

\ a) in the case $X \subset (0, \infty)$ we have $$U(x) :=
\frac{\text{min}\{1, x\}^{1-\sigma}}{1-\sigma} +
\frac{(2-\text{min}\{1, x\})^{1-\gamma}}{1-\gamma}, \ x \in X, \
\text{with} \ \sigma >0, \gamma>0.$$

 \ b) in the case $X \subset (0, 1) \times (0, 1)$, we have $$ U(x, y):=
\frac{x^{1-\sigma}}{1-\sigma} + \frac{y^{1-\gamma}}{1-\gamma}, \
(x, y) \in X, \ \text{with} \ \sigma>0, \gamma >0.$$
\end{defn}

The discount factor in the definition of $W$ expresses the fact that future levels of consumption in intertemporal equilibria become less and less relevant to a representative consumer. \
In economic models with backward dynamics we form as before the set of intertemporal equilibria i.e the inverse limit $\hat \Lambda$, where $f|_\Lambda : \Lambda \to \Lambda$ is the restriction of the dynamical system $f$ to a compact invariant set $\Lambda$. In general $f$ is assumed hyperbolic on $\Lambda$ or conjugated to a subshift of finite type of 1-sided sequences.
The consumers/agents have a utility function $W$ given on $\hat \Lambda$. A central government would like to know the average value of $W$ over $\hat \Lambda$. The question is \textbf{with respect to which measure on $\hat \Lambda$}?

In general one uses probability measures which are preserved by the system (in fact from any arbitrary probability measure we can form an invariant one, according to Krylov-Bogolyubov procedure, see \cite{KH}). \
Now an intertemporal equilibrium $\hat x \in \hat \Lambda$ represents in fact a sequence of future levels of consumption allowed by the implicit difference equations of our economic model. In reality an agent may preffer some open sets of intertemporal equilibria over others, and thus not all equilibria will have the same weight/importance, so it is important to use invariant probability measures $\hat \mu$ on the space $\hat \Lambda$ of intertemporal equilibria.
Also if we denote by $B_n(\hat x, \vp)$ the set of points $\hat y \in \hat \Lambda$ which $\vp$-shadow the orbit of $\hat x$ up to $n$-th iterate (called also a \textit{Bowen ball} in $\hat \Lambda$), we would like to have the measure $\hat \mu$ of $B_n(\hat x, \vp)$ as large as possible. This means we keep the disorder in the system as small as possible, and is equivalent to: as small an entropy $h_{\hat \mu}$ as possible.  Indeed it can be shown in general (Brin-Katok Theorem, see \cite{Ma}) that if $\mu$ is an $f$-invariant ergodic measure on a space $X$, then
for $\mu$-almost all $x \in X$, $$h_{\mu} = \mathop{\lim}\limits_{\vp \to 0} \lim \frac 1n \mu(B_n(x, \vp))$$

For instance in the cash-in-advance model (see \cite{KS},
\cite{KSY}, \cite{S}, etc) the government is controlling controls
the money supply on the market by the growth rule $M_{t+1} =
(1+\theta)M_t$, where $\theta>0$ is the growth rate. For each
$\theta$ there exists a different invariant interval
$[x_l(\theta), x_r(\theta)]$ and inverse limit space $\hat
\Lambda(\theta)$. For a utility function $W$ like
in Definition \ref{ut}, economists are interested also in choosing
the appropiate $\theta$ so that the average value $\int_{\hat
\Lambda(\theta)} W d\hat \mu_\theta$ is largest, where $\hat \mu$
is an invariant probability on $\hat \Lambda(\theta)$. In this way
given a certain utility function, we can adjust the money growth rate $\theta$ in such a
way that the average utility value is largest. Many times we want
to study systems from the point of view of the measure of maximal
entropy, which best describes the chaotic nature of the model. Also one can be interested in adjusting the discount
factor $\beta$ of $W$ in order to maximize the average utility value.

We will say below that a compact invariant set $\Lambda$ is \textbf{basic} for $f$ if there exists an open neighbourhood $V$ of $\Lambda$ s.t $\Lambda = \mathop{\cap}\limits_{n \in \mathbb Z} f^n(V)$ and if $f$ is topologically (forward) transitive on $\Lambda$; such a set is also called \textit{locally maximal} (see \cite{KH}). In general the invariant limit sets we have considered in the economic models so far, are basic by construction.

Let us recall the following result about invariant measures on inverse limits (see for instance \cite{Ru-T}); recall that our hyperbolic case includes also the expanding case.

\begin{unthm} [Invariant Measures on Inverse Limits]
Let $f: \Lambda \to \Lambda $ be a continuous topologically transitive map on a compact metric space $\Lambda$ and let $\hat f: \hat \Lambda \to \hat \Lambda$ be its inverse limits. Then there is a bijective correspondence $\mathcal{F}$ between $f$-invariant measures on $\Lambda$ and $\hat f$-invariant measures on $\hat \Lambda$, given by $\mathcal{F}(\hat \mu) = \pi_* (\hat \mu)$ (where $\pi: \hat \Lambda \to \Lambda, \pi(\hat x) = x$ is the canonical projection).

Moreover if in addition $f$ is hyperbolic on the basic set $\Lambda$, then for any Holder continuous potential $\phi$ on $\Lambda$ there exists a unique equilibrium measure $\hat \mu_{\phi\circ \pi}$ of $\phi\circ \pi$ and $\pi_*(\hat \mu_{\phi \circ \pi}) = \mu_\phi$, where $\mu_\phi$ is the equilibrium measure of $\phi$ on $\Lambda$.
\end{unthm}

We give now a formula for the average value of the utility with respect to \textit{any} invariant measure on the inverse limit.

\begin{thm}\label{inv}
Consider a continuous non-invertible map $f$ defined on an open set $V$ in $\mathbb R^2$ or in $\mathbb R$, which has an invariant basic set $\Lambda$.  Let also $W(\hat x) = \mathop{\sum}\limits_{i \ge 0} \beta^i U(x_{-i})$ be a utility function on the inverse limit $\hat \Lambda$ as in Definition \ref{ut}. Then for any $\hat f$-invariant borelian measure $\hat \mu$ on $\hat \Lambda$ we have that the average value $$\int_{\hat \Lambda} W d\hat \mu = \frac{1}{1-\beta} \int_\Lambda U d\mu,$$
where $\mu = \pi_*(\hat \mu)$.
If in addition $f$ is hyperbolic on $\Lambda$ and if $\mu_0$ is the unique $f$-invariant measure of maximal entropy on $\Lambda$ and $\hat \mu_0$ is the unique measure of maximal entropy on $\hat \Lambda$, then $\mu_0 = \pi_*(\hat \mu_0)$ and $\int_{\hat \Lambda} W d\hat \mu_0 = \frac{1}{1-\beta} \int_{\Lambda} U d\mu_0$.
\end{thm}

\begin{proof}
If we take the approximating functions $W_n(\hat x) =
\mathop{\sum}\limits_{i=0}^n \beta^i U(x_{-i})$, then $W_n$ converge uniformly towards $W$ 
since $||W-W_n|| \le C \beta^n, n \ge 1$. Hence 
$\int_{\hat\Lambda} W_n d\hat \mu  \mathop{\to} \limits_{n \to \infty}
\int_\Lambda W_n d \hat \mu$. Now recall that the measure $\hat
\mu$ is $\hat f$-invariant hence $$\int_{\hat \Lambda} W_n d\hat
\mu = \int_{\hat \Lambda} W_n\circ \hat f^n d\hat \mu = \int_{\hat
\Lambda} U(f^nx) + \beta U(f^{n-1} x) + \ldots + \beta^n U(x)
d\hat \mu$$

But now from the fact that $\mu = \pi_* (\hat \mu)$ we see that
$\int_{\hat \Lambda} g\circ \pi d\hat \mu = \int_{\Lambda} g d \mu$,
if $g$ is any continuous function on $\Lambda$. From the $f$-invariance of $\mu$ we have
$\int_\Lambda U\circ f^i d\mu = \int_\Lambda U d\mu, i \ge 0$; thus in our case $$\int_{\hat
\Lambda} W_n d\hat \mu = \int_\Lambda U(f^n x) + \ldots + \beta^n
U(x) d\mu(x) = (1 + \beta + \ldots + \beta^n) \int_\Lambda U(x)
d\mu(x)$$ So
from the approximation above, we obtain in conclusion that $$
\int_{\hat \Lambda} W d\hat \mu = \frac{1}{1-\beta} \int_\Lambda U
d\mu$$

In particular from the Theorem on Equilibrium Measures above, we
obtain that the unique measure of maximal entropy on $\Lambda$ is
the projection of the unique measure of maximal entropy on $\hat
\Lambda$, i.e $\mu_0 = \pi_*(\hat \mu_0)$ and from the above, $\int_{\hat \Lambda} W d\hat \mu_0 = \frac{1}{1-\beta} \int_{\Lambda} U d\mu_0$.

\end{proof}

If we consider $\mathcal{C}^2$-perturbations $g$ of a hyperbolic endomorphism $f$ on a basic set $\Lambda$ (including the case of a perturbation of an expanding endomorphism on a basic set), then from Theorem \ref{perturbation} we see that there exists a $g$-invariant basic set $\Lambda_g$ s.t $g$ is hyperbolic on $\Lambda_g$ and there exists a conjugating homeomorphism $H: \hat \Lambda \to \hat \Lambda_g$ with $\hat g \circ H = H \circ \hat f$. Then the measure of maximal entropy on $\hat \Lambda_g$, denoted by $\hat \mu_{0, g}$, is obtained as $H_*(\hat \mu_0)$, where $\hat \mu_0$ is the unique measure of maximal entropy on $\Lambda$. \
Thus in general we an calculate the average value of the utility $W$ with respect to the measure of maximal
entropy $\int_{\hat \Lambda_g} W d\hat \mu_{0, g}$ by applying Theorem \ref{inv} and the fact that $\mu_{0, g} = (\pi_g\circ H \circ \hat f)_*(\hat \mu_0)$, i.e
$$\int_{\hat \Lambda_g}W d\hat \mu_0 = \frac{1}{1-\beta} \int_{\Lambda_g} U d (\pi_g\circ H \circ \hat f)_*(\hat \mu_0)$$

The average values of $U$ on  $\hat\Lambda_g$
with respect to the corresponding measures of
maximal entropy, are easier to estimate than those on inverse
limits.  Economists can use this information to compare average
utility values with respect to the corresponding measures of
maximal entropy for various perturbations, which in reality are
translated by adjustments of the money growth rates.

A case in which this average utility ranking can be applied nicely is for the 1-dimensional overlapping generations economic model in which the backward dynamics is given by a Type C unimodal map (typically the \textit{logistic function} $F_\nu(x) = \nu x(1-x)$ with $\nu > 4$). In this case a central government can choose \textbf{both} the $\nu$ and the $\beta$ which \textbf{maximize the average utility value} over the set of intertemporal equilibria, with respect to the \textbf{measure of maximal entropy} (i.e the invariant measure describing the chaotic distribution over time).

\begin{cor}\label{log}
Let a family of logistic maps given by $F_\nu(x) = \nu x(1-x), x
\in [0, 1]$ with $\nu >4$; then $F_\nu$ has an invariant Cantor
set $\Lambda_\nu$. Consider also a utility function $W_\beta(\hat x) =
\mathop{\sum}\limits_{i \ge 0} \beta^i U(x_{-i})$ with $U(x) :=
\frac{\text{min}\{1, x\}^{1-\sigma}}{1-\sigma} +
\frac{(2-\text{min}\{1, x\})^{1-\gamma}}{1-\gamma}, \ x \in (0,
1), \ \text{for some} \ \sigma
>0, \gamma>0$. Then $$\int_{\hat \Lambda_\nu} W_\beta d\hat \mu_0 =
\frac{1}{1-\beta} \int_{\Sigma_2^+} U\circ h_\nu^{-1} d \mu_{\frac
12, \frac 12},$$ where $\hat \mu_0$ is the measure of maximal
entropy on $\hat\Lambda_\nu$, $\mu_{\frac 12, \frac 12}$ is the
measure of maximal entropy on $\Sigma_2^+$ and $h_\nu :
\Lambda_\nu \to \Sigma_2^+$ is the itinerary map, i.e $h_\nu(x) =
(j_0, j_1, \ldots)$ s.t $F_\nu^k(x) \in I_{j_k}, k \ge 0$ where
$F_\nu^{-1}([0, 1]) = I_1 \cup I_2, \ I_1 \cap I_2 = \emptyset$.
\end{cor}

\begin{proof}
For the logistic map $F_\nu$ with $\nu >4$ it is well known  (see
for instance \cite{R}) that $F_\nu$ has an invariant Cantor set
$\Lambda_\nu$. For $\nu> 2+\sqrt{5}$ the map $F_\nu$ is expanding
in the Euclidean metric, and for $4 < \nu \le 2+ \sqrt{5}$, the
map $F_\nu$ is expanding in a modified metric.

Also recall that $F_\nu^{-1}([0, 1]) = I_1 \cup I_2 \subset [0,
1]$ where the subintervals $I_1, I_2$ are disjoint. Then we have
the itinerary map $h_\nu: \Lambda_\nu \to \Sigma_2^+, h(x) = (j_0,
j_1, \ldots)$ given by $F_\nu^k(x) \in I_{j_k}, k \ge 0, x \in
\Lambda_\nu$. It can be noticed that $h_\nu$ is a homeomorphism
which gives the conjugacy between $F_\nu|_{\Lambda_\nu}$ and
$\sigma_2|_{\Sigma_2^+}$.

Now consider the measure of maximal entropy
 $\mu_{\frac 12, \frac 12}$ on $\Sigma_2^+$; we know (see for instance \cite{KH}) that $\mu_{\frac 12, \frac 12}$  gives measure
 $\frac {1}{2^k}$ to each of the cylinders $\{\hat\omega = (i_0, \ldots, i_{k-1}, j_{k}, \ldots), j_{k}, \ldots
 \in \{1, 2\}\}$ when $i_0, \ldots, i_{k-1}$ are fixed, ranging in $\{1, 2\}$.

From the conjugacy above,  $h_\nu^{-1}$ transports the measure of maximal entropy
 $\mu_{\frac 12, \frac 12}$ on $\Sigma_2^+$ to the measure of maximal entropy $\mu_0$ on $\Lambda_\nu$, i.e
 $(h_\nu^{-1})_*(\mu_{\frac 12, \frac 12}) =
 \mu_0$.
 And from the Theorem on Invariant Measures on Inverse Limits above, we know that
 $\mu_0 = \pi_*(\hat \mu_0)$, where $\hat \mu_0$ is the unique measure of maximal entropy on $\hat
 \Lambda_\nu$. \
So by applying Theorem \ref{inv} we obtain that $$\int_{\hat
\Lambda_\nu} W_\beta d\hat \mu_0 = \frac{1}{1-\beta}
\int_{\Sigma_2^+} U \circ h_\nu^{-1} d \mu_{\frac 12, \frac 12}$$
\end{proof}

Since we have an expression for the itinerary map
$h_\nu$  not difficult to approximate, and since the
measure $\mu_{\frac 12, \frac 12}$ is relatively easy to work
with, one can use Corollary \ref{log} to find a pair of
parameters $(\nu, \beta)$ maximizing the \textit{average
utility value with respect to the measure of maximal entropy} $$\int_{\hat \Lambda_\nu} W_\beta(\hat x)
d\hat\mu_0(\hat x)$$

\

We will now consider the \textbf{second ranking option} for
utility functions, i.e with respect to their equilibrium measures.
First we give some general topological dynamics definitions and
results.

\begin{defn}\label{expans}
A homeomorphism $f:X \to X$ on a metric space $X$ is called
\textit{expansive} if there exists a positive constant $\delta_0$
s.t if $d(f^i x, f^i y) < \delta_0, i \in \mathbb Z$ then $x = y$.
\end{defn}

The following property is very important for the existence of equilibrium measures of Holder continuous potentials (see \cite{Bo}, \cite{KH}).

\begin{defn}\label{spec}
Let a metric space $X$ and a continuous map $f:X \to X$. A
specification $S = (\tau, P)$ consists of a finite collection
$\tau = \{I_1, \ldots, I_m\}$ of finite intervals $I_i = [a_i,
b_i] \subset \mathbb Z$ and a map $P:T(S)=\mathop{\cup}\limits_{i
=1}^m I_i \to X$ s.t for any $t_1, t_2 \in I_j \in \tau$, we have
$f^{t_2 - t_1}(P(t_1)) = P(t_2)$. The specification $S$ is said to
be $n$-spaced if $a_{i+1} > b_i +n, 1 \le i \le m$ and the minimal
such $n$ is called the spacing of $S$. Let us denote also by $L(S)
= b_m - a_1$. We say that $S$ is $\vp$-shadowed by a point $x \in
X$ if $d(f^n(x), P(n)) < \vp$ for all $n \in T(S)$; if $T(S)$
contains also negative integers, we shadow with iterates of a
preimage of large order of $x$. The map $f$ has the
\textbf{specification property} if for any $\vp>0$ there exists an
$M = M_\vp \in \mathbb N$ s.t any $M$-spaced specification $S$ is
$\vp$-shadowed by a point of $X$ and for any $q \ge M + L(S)$,
there is a period-$q$ orbit $\vp$-shadowing $S$.
\end{defn}

\textbf{Remark.} In the above Definition, if $x$ is the period-$q$ point used in the shadowing and if $a_1 < 0$, then instead of $f^{a_1}(x)$ we can take $f^{kq+a_1}(x)$, for the smallest integer $k \ge 0$ s.t $0 \le kq+a_1 < q$  (as the map is non-invertible); then use forward iterates of this point $f^{kq+a_1}(x)$ in the shadowing of the specification.$\hfill\square$

 Let us consider now a continuous map $f:X \to X$ on a metric space $X$ and its inverse limit $(\hat X, \hat f)$, where $\hat X$ is the space of infinite sequences of consecutive preimages and $\hat f: \hat X \to \hat X$ is the shift homeomorphism.
In the sequel we will consider mixing basic sets $\Lambda$, i.e basic sets for the endomorphism $f$ s.t $f$ is topologically mixing on $\Lambda$. In fact from the Spectral Decomposition Theorem (see \cite{Bo}, \cite{KH}), any basic set can be decomposed into a finite partition $\Lambda_1, \ldots, \Lambda_s$ s.t for each $j$ there is some iterate $f^{k_j}$ which leaves $\Lambda_j$ invariant and which is topologically mixing on $\Lambda_j$.

\begin{thm}\label{exp}
Let us consider one of the examples from Section 1 that has a mixing basic set $\Lambda$ on which $f$ is hyperbolic.
Then the shift homeomorphism $\hat f$ is
expansive and has specification property on the inverse limit $\hat \Lambda$.
\end{thm}

\begin{proof}

First of all let us show that $\hat f$ is expansive on $\hat
\Lambda$. Let $\hat x, \hat y \in \hat \Lambda$ s.t $d(\hat f^i
\hat x, \hat f^i \hat y) < \delta$ for all $i \in \mathbb Z$ and
some small $\delta>0$. Now $f$ is hyperbolic as an endomorphism on
$\Lambda$ which from construction is a locally maximal set, i.e
there exists a neighbourhood $U$ of $\Lambda$ s.t $\Lambda =
\mathop{\cap}\limits_{n \in \mathbb Z} f^n(U)$. Then if $d(\hat
f^i \hat x, \hat f^i \hat y) < \delta, i \in \mathbb Z$, it
follows that $d(f^i x, f^i y) < \delta, i \ge 0$, hence $y \in
W^s_\delta(x)$. On the other hand if $d(x_{-i}, y_{-i}) < \delta,
i \ge 0$, for certain prehistories $\hat x, \hat y \in \hat
\Lambda$, it follows that $y \in W^u_\delta(\hat x)$. Now if
$\Lambda$ is a hyperbolic locally maximal set for $f$ it follows
that it has local product structure (see \cite{KH}); thus
$W^s_\delta(x) \cap W^u_\delta(\hat x) = \{x\}$ for $\delta>0$
small enough, so $x = y$. By repeating this argument for all
preimages $x_{-i}$ we obtain that $x_{-i} = y_{-i}, i \ge 0$.
Therefore $\hat x = \hat y$, and $\hat f$ is expansive on $\hat
\Lambda$.

Let us prove now that $\hat f$ has the specification property on $\hat \Lambda$.
We assumed that $f$ is hyperbolic and topologically mixing on $\Lambda$. 
Then  as in Theorem 18.3.9 of \cite{KH} we can adapt the proof to endomorphisms to show that $f$ has specification property on $\Lambda$.

In order to prove that $\hat f$ has the specification property on $\hat \Lambda$, let us consider a specification
$\hat S$ in $\hat \Lambda$,
$\hat S = (\hat\tau, \hat P)$, where $\hat \tau$ is a collection of finitely many intervals in
$\mathbb Z$ and $\hat P$ is a correspondence between $T(\hat \tau)$ and $\hat \Lambda$.
Assume that $\hat \tau = \{I_1, \ldots, I_m\}$, with $I_i = [a_i, b_i]$ and that
$\hat P(a_i) = \hat \omega^i = (\omega^i, \omega^i_{-1}, \ldots) \in \hat \Lambda, 1 \le i \le m$.

Consider a small $\vp>0$. We will construct now a
specification $S$ in $\Lambda$ with bigger intervals than those of
$\hat S$. Assume that $\text{diam}(\Lambda) \le 1$ and
take $r = r(\vp)$ so large that $\frac{1}{2^{r}} < \vp/2$. Then we
see that if $d(f^j(x_{-r}), f^j(y_{-r})) < \vp/4, 0 \le j \le r$,
then $d(\hat x, \hat y) < \vp$, where $\hat x = (x, x_{-1},
\ldots), \hat y = (y, y_{-1}, \ldots)$. Hence consider the
specification $S$ in $\Lambda$ of the form $(\tau, P)$, where
$\tau = \{[a_1 -r, b_1] \ldots, [a_m -r, b_m]\}$ and $P(a_i-r) =
\omega^i_{-r}, \ldots, P(b_i) = f^{b_i - a_i}(\omega^i), 1 \le i
\le m$. If $a_1 -r<0$ then instead of $f^{a_1-r}(p)$ we take in the shadowing the iterate $f^{kq+a_1-r}(p)$, for the smallest integer $k \ge 0$ s.t $kq+a_1-r \in [0, q)$. For the other points in the orbit of $p$ used for shadowing we take the
positive  iterates  of $f^{kq+a_1-r}(p)$, i.e $d(\omega^1_{-r+1}, f^{kq+a_1-r+1}(p)) < \vp/4$, etc.

Now assume that the specification $\hat S$ is $(M + r)$-spaced,
where $M=M(\vp/4)$ is the spacing from the specification property
of $f|_\Lambda$ corresponding to $\vp/4$, and where $r = r(\vp)$
is given above. Then from the specification property of $f$ on
$\Lambda$ it follows that for $q \ge M+L(S)= M + L(\hat S) + r$ there is a period-$q$
orbit $\{p, f(p), \ldots, f^{q-1}(p)\}$ which $\vp/4$-shadows $S$.
Then for $r = r(\vp)$ we can take $\hat M(\vp):= M(\vp/4) + r$,
and the orbit of the periodic point of period $q$, $$\hat p= (f^{kq+a_1-r}(p),
f^{kq+a_1-r-1}(p),  \ldots, p, \ldots, f^{kq+a_1-r}(p), \ldots) \in \hat \Lambda$$ We know from
the construction of $S$ that the orbit of $f^{kq+a_1-r}(p)$, $\vp/4$-shadows the composite
chain of points $$\{\omega^1_{-r}, \ldots, \omega^1, \ldots
f^{b_1-a_1}(\omega^1)\} \cup \ldots \cup \{\omega^m_{-r}, \ldots,
\omega^m, \ldots, f^{b_m - a_m}(\omega^m)\}$$

Thus we have $d(\omega^1_{-r}, f^{kq+a_1-r}(p))< \vp/4, \ldots, d(\omega^1,
 f^{kq+a_1}(p)) < \vp/4, \ldots, d(f^{b_1 - a_1}(\omega^1), f^{kq+b_1}(p)) < \vp/4$ and so on up to the interval $I_m$ where  $d(\omega^m_{-r}, f^{kq+a_m-r}(p))< \vp/4, \ldots, d(\omega^m, f^{kq+a_m}(p)) < \vp/4, \ldots, d(f^{b_m - a_m}(\omega^m), f^{kq+b_m}(p))   < \vp/4$. 

We want to prove that the orbit of  $\hat p$, $\vp$-shadows the
specification $\hat S$. From above we obtain that 
$$
\aligned
d(\hat
\omega_i, \hat f^{a_i}(\hat p)) &= d(\omega^i, f^{kq+a_1-r+a_i}(p)) +
\frac{d(\omega^i_{-1}, f^{kq+a_1-r +a_i-1}(p))}{2} + \ldots +
\frac{d(\omega^i_{-r}, f^{kq+a_1-r+a_i-r} (p))}{2^r} + \ldots \\
&< \vp/4 + \vp/8
+ \vp/2^{r+2} + \frac{1}{2^r} < \vp/2 + \vp/2 = \vp,
\endaligned
$$
 which
follows from the way we chose $r$ above, i. e such that $\frac{1}{2^r} <
\vp/2$.
  \ Then we can similarly prove these inequalities up to order $b_i$ when:
$$
\aligned 
d(\hat f^{b_i - a_i} \hat \omega^i, \hat f^{b_i} \hat p) &= d(f^{b_i-a_i}(\omega^i), f^{kq+a_1-r+b_i}(p)) +
  \ldots + \frac{d(f^{b_i - a_i}(\omega^i_{-r}), f^{kq+a_1-r+b_i-r} p)}{2^r} + \ldots \\
&<  \vp/4 + \vp/8 + \ldots + \vp/2^{r+2} +
  \frac{1}{2^r} < \vp
\endaligned
$$
 Since the above estimates can be done for all $i = 1, \ldots, m$ we see that
  $\hat p$, $\vp$-shadows the specification $\hat S$ if $\hat S$ is $\hat M(\vp):= (M(\vp/4) + 2r)$-spaced.

We notice that the integer $r = r(\vp)$ does not depend on the
specification $\hat S$; in conclusion for any $\vp>0$ we found a
positive integer $\hat M(\vp)$ so that any $\hat M(\vp)$-spaced
specification $\hat S$ in $\hat \Lambda$ is $\vp$-shadowed by a
point in $\hat \Lambda$, and for any $q \ge \hat M(\vp) + L(\hat
S)$ there exists a period-$q$ orbit $\vp$-shadowing $\hat S$.

In conclusion  if $f$ has specification property on $\Lambda$, then
also  $\hat f$ has specification property on $\hat \Lambda$ which
finishes the proof of the Theorem.

\end{proof}

A representative agent may want to maximize the average value of
his utility function with respect to a $\hat f$-invariant measure
$\hat \mu$ on $\hat \Lambda$ but \textit{at the same time} to have
as much control on the system as possible in the long run. In
other words a possibility is to maximize the following sum giving
the average value plus the control $h_{\hat \mu}$:
\begin{equation}\label{max}
AC(W) (\hat \mu) = \int_{\hat \Lambda} W d\hat \mu + h_{\hat \mu}
\end{equation}

From the Variational Principle for Topological Pressure (see
\cite{KH} for eg.), we know that $AC(W)(\hat \mu)$ is maximized
for a probability measure called the \textbf{equilibrium measure}
of $W$. If $W$ is Holder continuous and $\hat f$ is expansive then
this measure is unique and will be denoted by $\hat \mu_W$. This
measure has important geometric properties and one can precisely
estimate the measure $\hat \mu_W$ of the \textit{Bowen balls}
$B_n(\hat x, \vp):= \{\hat y \in \hat \Lambda, d(\hat f^i \hat y,
\hat f^i \hat x) < \vp, i = 0, \ldots, n-1\}$ (see for eg. \cite{Bo}, \cite{KH}).

In particular when $W$ is constant, the equilibrium measure of $W$
is the measure of maximal entropy. Equilibrium measures appear
also as \textit{Sinai-Ruelle-Bowen measures} in the case of
hyperbolic attractors (see \cite{ER}, \cite{Y}) which give the
limiting distribution of forward trajectories of Lebesgue-almost
all points in a neighborhood of the attractor. In the case of
\textit{non-invertible hyperbolic repellers} equilibrium measures
of stable potentials appear also as \textit{inverse
Sinai-Ruelle-Bowen measures} (see \cite{M-JSP}), i.e invariant
measures describing the limiting distributions of preimages of
large orders, of Lebesgue almost-all points in a neighbourhood of
the non-invertible repeller.

We have the following Theorem giving the measure of a Bowen ball
$B_n(x, \vp)$ in a metric space (see \cite{KH}); by $S_n\phi(y)$
we denote the \textbf{consecutive sum} $\phi(y) + \phi(f(y)) +
\ldots + \phi(f^{n-1}(y))$.

\begin{unthm}[Bowen's Theorem on Equilibrium Measures.]
Let $(X, d)$ be a compact metric space and $f: X \to X$ an expansive homeomorphism with specification property and
 $\phi: X \to \mathbb R$ a Holder continuous potential on $X$. Then there exists exactly one equilibrium measure for
  $\phi$ and $$\mu_\phi = \mathop{\lim}\limits_{n \to \infty} \frac{1}{\mathop{\sum}\limits_{y \in
  \text{Fix}(f^n)}
   e^{S_n\phi(y)}} \mathop{\sum}\limits_{y \in \text{Fix}(f^n)} e^{S_n\phi(y)} \delta_y$$
Moreover we can estimate the measure $\mu_\phi$ of Bowen balls by:
\begin{equation}\label{Bw}
A_\vp e^{S_n\phi(y)-nP(\phi)} \le \mu_\phi(B_n(y, \vp)) \le B_\vp e^{S_n\phi(y) - nP(\phi)}, \ y \in X, n \ge 1,
\end{equation}
where $A_\vp, B_\vp>0$ are positive constants depending only on $\vp$, and $P(\phi)$ is a number called the topological pressure of $\phi$.
\end{unthm}

Now we notice that in the examples from Section 1 presenting a hyperbolic set, they are formed from non-critical homoclinic orbits to repelling fixed points or from horseshoes without critical points.

\begin{thm}\label{eq}
Consider one of the economic systems from Section 1 given by a
non-invertible map $f$ that has a hyperbolic mixing basic set
$\Lambda$ containing no critical points of $f$. Let also a utility
function $W$ defined on the inverse limit space $\hat \Lambda$ as
in Definition \ref{ut}. Then there exists a unique equilibrium
measure $\hat \mu_W$ of $W$ on $\hat \Lambda$ and for any $\vp>0$
there are positive constants $A_\vp, B_\vp$ so that for any $ \hat
x \in \hat \Lambda, n \ge 1$, $$A_\vp e^{S_nW(\hat x) - nP(W)} \le
\hat \mu_W(B_n(\hat x, \vp)) \le B_\vp e^{S_nW(\hat x)-
nP(W)}$$

\end{thm}

\begin{proof}
Let us consider the hyperbolic non-invertible map $f$ restricted
to the compact invariant set $\Lambda \subset \mathbb R^2$ having
an inverse limit $\hat \Lambda$, and $W$ as in Definition
\ref{ut} (the same proof works in the 1-dimensional case). The utility function $W$ has an associated discount
factor $\beta \in (0, 1)$.

We will show that $W(\hat x) = \mathop{\sum}\limits_{i \ge 0}
\beta^i U(x_{-i})$ is Holder continuous on the
  metric space $\hat \Lambda$.  Let us notice first that for the utility functions of
Definition
  \ref{ut}, the function $U$ is Holder continuous. So there exists a constant $C>0$ and an exponent
   $\gamma\in (0, 1]$ s. t $|U(x)- U(y)| \le C d(x, y)^\gamma, x, y \in \Lambda$, as the set $\Lambda$ is
   compact.\
But $W(\hat x) = U(x) + \beta U(x_{-1}) + \beta^2 U(x_{-2}) +
\ldots$, so $|W(\hat x) - W(\hat y)| \le |U(x) - U(y)| + \beta
|U(x_{-1}) - U(y_{-1})| + \beta^2 |U(x_{-2}) - U(y_{-2})| +
\ldots, \ \hat x, \hat y \in \hat \Lambda$. From the Holder
condition for $U$ we obtain that $|U(x_{-i}) - U(y_{-i})| \le C
d(x_{-i}, y_{-i})^\gamma, i \ge 0$. Hence
\begin{equation}\label{W}
|W(\hat x) - W(\hat y)| \le C \cdot [d(x, y)^\gamma + \beta d(x_{-1}, y_{-1})^\gamma + \ldots], \hat x, \hat y \in \hat \Lambda
\end{equation}

Without loss of generality assume that $\text{diam}(\Lambda) = 1$.
Let us take now two close points $\hat x, \hat y \in \hat \Lambda,
d(\hat x, \hat y) < \delta < < 1$. Recall that we  have a
hyperbolic structure on $\Lambda$, and denote by $Df_s(x)$ the
restriction of $Df(x)$ to the stable tangent space at $x$. If $x
\ne y$ are close, then we may have some of their preimages of order 1, $x_{-1}$ and
$y_{-1}$ close as well. Denote by $\lambda:= \frac{1}{\inf_\Lambda
|Df_s|}$; then $1 < \lambda < \infty$ since there are no critical
points in $\Lambda$. Assume also that $\gamma>0$ is taken such
that:
\begin{equation}\label{con}
\beta \lambda^\gamma <1
\end{equation}
This is possible if we take $\gamma>0$ small enough, since $\beta
\in (0, 1)$. From the definition of $\lambda$, we know that
$d(x_{-1}, y_{-1}) \le d(x, y) \lambda$ if $x_{-1}, y_{-1}$ are close too. Let us repeat this
procedure with finite sequences of consecutive preimages $x_{-m}, y_{-m}$ until we have $d(x,
y) \lambda^m
>\vp_0$ for some fixed $\vp_0$; i.e $m$ is the first positive
integer satisfying this condition. Then for a choice of $\hat
x, \hat y$ having on the $m$-th positions respectively $x_{-m}, y_{-m}$, we obtain from (\ref{W}): $$|W(\hat x)- W(\hat y)| \le C[d(x,
y)^\gamma + \beta d(x, y)^\gamma \lambda^\gamma + \ldots + \beta^m
d(x, y)^\gamma \lambda^{m\gamma} + \beta^m]$$ We know however that
$m$ is related to $d(x, y)$ and can be expressed in terms of it.
Indeed from the condition on $m$, we have that $m\log \lambda \ge
\log \frac{\vp_0}{d(x, y)}$ and hence $$\beta^m \le C_1 \cdot d(x,
y)^{\rho'},$$ for some constant $\rho' >0$. This together with the
above relation mean that $$|W(\hat x) - W(\hat y)| \le
\frac{C}{1-\beta \lambda^\gamma} d(x, y)^\gamma + C_1 d(x,
y)^{\rho'}$$ So by taking $\rho := \text{min}\{\rho', \gamma\}$ we
obtain that $|W(\hat x) - W(\hat y)| \le C_2 d(x, y)^\rho$. But
$d(\hat x, \hat y) \ge d(x, y)$, therefore we obtain Holder
continuity in this case, namely $|W(\hat x, \hat y) | \le
C_2 d(\hat x, \hat y)^\rho$.

Now assume that $\hat x, \hat y$ are not as above i.e they do not
shadow each other up to order $m$ but instead, for some $1 \le j
\le m$ there is a preimage $y_{-j}$ far from $x_{-j}$, i.e
$d(x_{-j}, y_{-j}) > \vp_0$ (this follows from the fact that there
are no critical points of $f$ in $\Lambda$). Assume that $\kappa$
is the smallest such $j$. Then $$ \aligned |W(\hat x) - W(\hat y)|
& \le C \left[d(x, y)^\gamma + \beta \lambda^\gamma d(x, y) +
\ldots + \beta^\kappa \lambda^{\kappa \gamma} d(x, y)^\gamma +
\beta^\kappa\right] \\
 &\le \frac{C}{1- \beta \lambda^\gamma} d(x,
y)^\gamma + C_1 \beta^\kappa,
\endaligned
$$
 for some constants $C, C_1 >0$.
Assume first that $d(x, y)^\gamma \le \beta^\kappa$; then $|W(\hat
x) - W(\hat y)| \le C_2 \beta^\kappa$. But $d(\hat x, \hat y) \ge
\frac{d(x_{-\kappa}, y_{-\kappa})}{2^\kappa} \ge
\frac{\vp_0}{2^\kappa}$. Hence there is a sufficiently small
positive constant $\rho$ and a constant $C_3>0$ (both independent of $\hat x, \hat y$) such
that $|W(\hat x) - W(\hat y)| \le C_3 d(\hat x, \hat y)^\rho$. Now
if we have the other case, i.e $d(x, y)^\gamma \ge \beta^\kappa$,
then $$|W(\hat x)- W(\hat y)| \le C_2 d(x, y)^\gamma \le C_2
d(\hat x, \hat y)^\gamma$$

Hence we proved that $W$ is Holder continuous on $\hat \Lambda$,
i.e  there are positive constants $C>0, \rho>0$ so that for all
$\hat x, \hat y \in \hat \Lambda$ we have $$|W(\hat x) - W(\hat
y)| \le C d(\hat x, \hat y)^\rho$$
Now we can use Theorem \ref{exp} in order to prove that the
homeomorphism $\hat f$ is expansive and has specification property
on $\hat \Lambda$. Since we showed that $W$ is Holder
continuous on $\hat \Lambda$ it follows that it has a unique
equilibrium measure $\hat \mu_W$ for which we have the estimates
on the measure of Bowen balls from the previous Bowen's Theorem.
Thus for any $\vp>0$ there are positive constants $A_\vp, B_\vp$
so that for $ \hat x \in \hat \Lambda, n \ge 1$, $$A_\vp
e^{S_nW(\hat x) - nP(W)} \le \hat \mu_W(B_n(\hat x, \vp)) \le
B_\vp e^{S_nW(\hat x)- nP(\phi)}$$
\end{proof}

The previous Theorem gives us good estimates for the measure $\hat
\mu_W$ of the set of points whose iterates remain close to the trajectory of a
certain initial condition, up to $n$ consecutive iterates.

We show now that, if we consider the measure of maximal entropy
$\hat \mu_0$ and \textbf{compare} it to the equilibrium measure
$\hat \mu_W$ on $\hat \Lambda$, then the average utility with
respect to $\hat \mu_W$ is bigger than the average utility with
respect to $\hat \mu_0$.

\begin{cor}\label{ut-en}
In the setting of Theorem \ref{eq} consider the measure of maximal
entropy of $\hat f$ on $\hat \Lambda$ and the
 equilibrium measure $\hat \mu_W$ of $W$ on $\hat \Lambda$. Then $$
\int_{\hat \Lambda} W d\hat \mu_W \ge \int_{\hat \Lambda} W d\hat \mu_0$$
\end{cor}

\begin{proof}
From the Variational Principle for topological pressure we know
that $\sup \{h_\nu + \int_{\hat \Lambda} W d\nu, \ \nu \ \hat
f-\text{invariant probability on} \ \hat \Lambda\} = P(W) = h_{\hat
\mu_W} + \int_{\hat \Lambda} W d\hat \mu_W$. Hence since $h_{\hat \mu_0} = h_{top}(\hat f)$ we obtain $$ \int_{\hat
\Lambda} W d\hat \mu_0 + h_{top}(\hat f) \le \int_{\hat \Lambda} W
d\hat \mu_W + h_{\hat \mu_W}$$ Then since $h_{top}(\hat f) \ge h_{\hat \mu_W}$
from the Variational Principle for Entropy (see \cite{KH}), we
obtain the conclusion of the Corollary.

\end{proof}

Given the specific form of our utility function, we can
approximate $\hat \mu_W$ with equilibrium states of simpler
functions. Consider $W_n(\hat x) = \mathop{\sum}\limits_{0 \le i
\le n} \beta^i U(x_{-i}),  \ \hat x \in \hat \Lambda, \ \text{for} \n \ge 1$. Similarly as in the proof
of Theorem \ref{eq} we can show that $W_n$ is a Holder function on
$\hat \Lambda$, hence it has an equilibrium state $\hat \mu_{W_n}$
on $\hat \Lambda$.

\begin{thm}\label{approx}
In the setting of Theorem \ref{eq}, let a utility function $W$ on $\hat \Lambda$ and the functions $W_n, n \ge 1$ as above. Then the average value of the utility function with respect to $\hat \mu_W$ can be approximated with those of $W_n$, i.e
$$
|\int_{\hat \Lambda} W d\hat \mu_W - \int_{\hat\Lambda}W_n d\hat \mu_{W_n}| \mathop{\to}\limits_{n \to \infty} 0
$$
\end{thm}

\begin{proof}
From Bowen's Theorem applied to equilibrium measures on $\hat \Lambda$ we have that $$\hat
\mu_\phi = \mathop{\lim}\limits_{n \to \infty}
\frac{1}{\mathop{\sum}\limits_{\hat x \in \text{Fix}(\hat f^n)}e^{S_n\phi(\hat x)}}
\mathop{\sum}\limits_{ \hat x \in \text{Fix}(\hat f^n)}
e^{S_n\phi(\hat x)} \delta_{\hat x},$$ for any Holder continuous
potential $\phi$ on $\hat \Lambda$. Hence since $$||W-W_n|| \le
\frac{\beta^n}{1-\beta} \sup_\Lambda |U|,$$ it follows that $n
\cdot |W-W_n|$ converges uniformly to $0$ and thus $\hat \mu_{W_n}
\to \hat \mu_W$ weakly. Hence $$ \aligned |\int W d\hat \mu_W -
\int W_n d\hat \mu_{W_n}| & \le |\int W d\hat \mu_W - \int W d\hat
\mu_{W_n}| + | \int W d\hat \mu_{W_n} - \int W_n d\hat \mu_{W_n}|
\\ &\le |\int W d\hat \mu_W - \int W d\hat \mu_{W_n}| + \frac{\beta^n}{1-\beta}
\cdot \sup_\Lambda |U|,\endaligned$$ since $||W- W_n|| \le
\frac{\beta^n}{1-\beta} \sup_\Lambda |U|$ and since $\hat
\mu_{W_n}$ is a probability measure. So from the weak
convergence of $\hat \mu_{W_n}$ towards $\hat \mu_W$, we obtain
the conclusion of the Theorem.

\end{proof}

\textbf{Acknowledgements:} This paper is suported by the Sectorial
Operational Programme Human Resources Development (SOP HRD),
financed from the European Social Fund and by the Romanian
Government under the contract number SOP HRD/89/1.5/S/62988.

\

\textbf{E-mail:}  Eugen.Mihailescu\@@imar.ro

Webpage: www.imar.ro/$\sim$mihailes

Institute of Mathematics ``Simion Stoilow`` of the Romanian
Academy, P.O. Box 1-764, RO 014700, Bucharest, Romania.


\begin{thebibliography}{99}

\bibitem{Bo}
R. \ Bowen, Equilibrium states and the ergodic theory of Anosov
diffeomorphisms, Lecture Notes in Mathematics, 470, Springer 1975.

\bibitem{ER}
J. P.\ Eckmann and D. \ Ruelle, Ergodic theory of strange
attractors, Rev. Mod. Physics, \textbf{57}, 1985, 617-656.

\bibitem{FG}
I.\ Foroni and L. \ Gardini, Homoclinic bifurcations in heterogeneous market models, Chaos, Solitons and Fractals, \textbf{15}, 2003, 743-760.

\bibitem{GHT}
L. \ Gardini, C. \ Hommes, F. Tramontana and R. \ de Vilder, Forward and backward dynamics in implicitly defined overlapping generations models, J. Economic Behaviour and Organization, \textbf{71}, 2009, 110-129.

\bibitem{G}
J.M. \ Grandmont, On endogeneous competitive business cycles, Econometrica, \textbf{53}, 1985, 995-1045.

\bibitem{HL}
J. \ Hale and X. Lin, Symbolic dynamics in nonlinear semiflows,
Ann. di Matem. Pura e Appl., \textbf{144}, 1986, 229-259.

\bibitem{KH}
A.\ Katok and B.\ Hasselblatt, Introduction to the Modern Theory
of Dynamical Systems, Cambridge Univ. Press, London-New York,
1995.

\bibitem{KS}
J. A \ Kennedy and D. R \ Stockman, Chaotic equilibria in models with backward dynamics, J. Economic Dynamics and Control, \textbf{32}, 2008, 939-955.

\bibitem{KSY}
J.A \ Kennedy, D. R\ Stockman and J. A \ Yorke, Inverse limits and an implicitly defined difference equation from economics, Topology and its Appl., \textbf{154}, 2007, 2533-2552.

\bibitem{LY}
T. Y \ Li and J.A\ Yorke, Period three implies chaos, American Math. Monthly, \textbf{82}, 1975, 985-992.

\bibitem{Ma}
R.\ Mane, Ergodic theory and differentiable dynamics, Springer Verlag, Berlin, New York, 1987.

\bibitem{Mar}
F. R \ Marotto, Snap-back repellers imply chaos in $\mathbb R^n$,
J. Math. Analysis and Applications \textbf{63}, 1978, 199-223.

\bibitem{Mar2}
F. R \ Marotto, On redefining a snap-back repeller, Chaos,
Solitions and Fractals, \textbf{25}, 2005, 25-28.

\bibitem{Med}
A. \ Medio and B. \ Raines, Backward dynamics in economics. The inverse limit approach, J. Econ. Dynamics and Control, \textbf{31}, 2007, 1633-1671.

\bibitem{Med2}
A. \ Medio and B. \ Raines, Implicit equilibrium dynamics, preprint 2007.

\bibitem{MR}
R. \ Michener and B. \ Ravikumar, Chaotic dynamics in a cash-in-advance economy, J. Econ. Dynamics and Control, \textbf{22}, 1998, 1117-1137.

%\bibitem{M-DCDS2001}
%E. \ Mihailescu, Applications of thermodynamic formalism in
%complex dynamics on $\mathbb P^2$, Discrete and Cont. Dynam.
%Syst., \textbf{7}, 2001, no.4, 821-836.

\bibitem{M-MZ}
E. \ Mihailescu, Unstable directions and fractal dimension for a
class of skew products with overlaps in fibers, Math. Zeitschrift
2010, DOI: 10.1007/s00209-010-0761-y.

\bibitem{M-JSP}
E. \ Mihailescu, Physical measures for multivalued inverse
iterates near hyperbolic repellors, J. Statistical Physics,
\textbf{139}, 2010, 800-819.

\bibitem{M-DCDS06}
E. \ Mihailescu, Unstable manifolds and H\"older structures
associated with noninvertible maps, Discrete and Cont. Dynam.
Syst. \textbf{14}, 3, 2006, 419-446.

%\bibitem{M-NA}
%E. \ Mihailescu, Ergodic properties for some non-expanding non-reversible systems, Nonlinear Analysis-Theory, Methods Applications, \textbf{73}, 2010, 3779-3787.

\bibitem{O}
T. \ Onozaki, G. \ Sieg and M. \ Yokoo, Complex dynamics in a cobweb model with adaptive production adjustment, J. Econ. Behabvior and Organization, \textbf{41}, 2000, 101-115.

%\bibitem{QZ}
%M. \ Qian, Z. \ Zhang, Ergodic theory for axiom A endomorphisms,
%Ergodic Th. and Dynam. Syst., \textbf{15}, 1995, 161-174.

%\bibitem{Ro}
%V.\ A.\ Rokhlin, Lectures on the theory of entropy of
%transformations with invariant measures, Russian Math. Surveys,
\textbf{22}, 1967, 1-54.

\bibitem{R}
C. \ Robinson, Dynamical Systems: Stability, Symbolic Dynamics and Chaos, CRC Press, Boca Raton, 1999.

\bibitem{Ru-carte89}
D. \ Ruelle, Elements of differentiable dynamics and bifurcation
theory, Academic Press, New York, 1989.

\bibitem{Ru-T}
D. \ Ruelle, Thermodynamic Formalism: The mathematical structures
of equilibrium statistical mechanics, Second Edition, Cambridge
Univ. Press, 2004.

\bibitem{S}
D. R\ Stockman, Uniform measures on inverse limit spaces, preprint Univ. Delaware, 2008.

\bibitem{Y}
L.\ S.\ Young, What are SRB measures, and which dynamical systems
have them? Dedicated to David Ruelle and Yasha Sinai on the
occasion of their 65th birthdays, J. Statistical Phys. {\bf 108},
2002, 733-754.

\bibitem{Z}
W. B \ Zhang, Discrete Dynamical Systems, Bifurcations and Chaos in Economics, Elsevier, 2006.

\end{thebibliography}
\end{document}